\documentclass[a4paper,10pt]{amsart}
\usepackage[utf8]{inputenc}
\usepackage{amsmath,amssymb,amsthm}
\usepackage{stmaryrd}
\usepackage{xcolor}
\usepackage{enumerate}
\usepackage{caption}
\usepackage[title]{appendix}

\usepackage{hyperref}

\theoremstyle{plain}

\newtheorem{thm}{Theorem}[section]
\newtheorem*{thm*}{Theorem}
\newtheorem{cor}[thm]{Corollary}
\newtheorem{lem}[thm]{Lemma}
\newtheorem{prop}[thm]{Proposition}

\theoremstyle{definition}

\renewcommand{\epsilon}{\varepsilon}

\newcommand{\comm}[1]{}

\numberwithin{equation}{section}

\title[Commuting and product-zero probability in finite rings]{Commuting and product-zero probability \\ in finite rings}
\author[P.\ Shumyatsky]{Pavel Shumyatsky}
\address{Department of Mathematics, University of Brasilia, Brasilia DF, Brazil}
\email{pavel@unb.br}
\author {Matteo Vannacci}
\address{Matematika Saila \\ 
UPV-EHU \\ Barrio Sarriena, s/n \\
48910 Bilbao, Spain}
\email{matteo.vannacci@ehu.eus}

\thanks{This paper was written while the first author was visiting the Department of Mathematics of the University of the Basque Country. He expresses his sincere gratitude to the department for excellent hospitality.
The second author has been supported by the Spanish Government, grants PID2020-117281GB-I00 and  PID2019-107444GA-I00, partly with FEDER funds, and the Basque Government, grant IT1483-22.}

\begin{document}

 \begin{abstract}
  Let $\mathrm{cp}(R)$ be the probability that two random elements of a finite ring $R$ commute and $\mathrm{zp}(R)$ the probability that the product of two random elements in $R$ is zero. We show that if $\mathrm{cp}(R)=\varepsilon$, then there exists a Lie-ideal $D$ in the Lie-ring $(R,[\cdot,\cdot])$ with $\varepsilon$-bounded index and with $[D,D]$ of $\varepsilon$-bounded order. If $\mathrm{zp}(R)=\varepsilon$, then there exists an ideal $D$ in $R$ with $\varepsilon$-bounded index and $D^2$ of $\varepsilon$-bounded order. These results are analogous to the well-known theorem of P. Neumann on the commuting probability in finite groups.
 \end{abstract}

\maketitle

 \section*{Introduction}

 In this article we are interested in the probability $\mathrm{cp}(R)$ that two random elements of a finite ring $R$ commute and in the probability $\mathrm{zp}(R)$ that the product of two random elements in $R$ is zero. Explicitly: for a finite ring $R$ we write
 \[
   \mathrm{cp}(R) = \frac{\lvert \{(x,y)\in R\times R \mid [x,y]=0\}\rvert}{\lvert R \rvert^2} \text{\  \ and\ \ }  \mathrm{zp}(R) = \frac{\lvert\{(x,y)\in R\times R \mid xy=0\}\rvert}{\lvert R \rvert^2}.
 \]
These probabilities are known as the \emph{commuting probability} and the \emph{zero-probability} of the ring $R$, respectively, and they have been studied in several articles (see \cite{machale,BM} for the commuting probability and \cite{D,EJ1,EJ2} for the zero-probability).

 Moreover, the same questions for finite groups have also been investigated. In the case of finite groups, the most famous results are surely Gustavson's 5/8th's Theorem \cite[Introduction]{gustafson} and P.\ Neumann's theorem \cite[Theorem~1]{neumann}.

 \begin{thm*}{\cite{gustafson}}
  Let $G$ be a finite group such that $\mathrm{cp}(G)\ge 5/8$. Then, $G$ is abelian.
 \end{thm*}

 \begin{thm*}{\cite[Theorem~1]{neumann}}
 Let $\varepsilon>0$ be a real number and $G$ be a finite group. If $\mathrm{cp}(G)\ge\varepsilon$, then $G$ has a normal subgroup $T$ such that the index $\lvert G:T\rvert$ and the order of the commutator subgroup $[T,T]$ are both $\varepsilon$-bounded.
 \end{thm*}

 It is somewhat surprising that, even though Gustavson's Theorem for rings was proved several years ago \cite{machale}, Neumann's theorem for finite rings was not yet known. In this article we complete this analogy.

 \section{Results}

 Whenever we talk about the index $[R:A]$ of a subring $A$ in a ring $R$, we mean the index of $(A,+)$ in $(R,+)$ as additive groups. 

 Our first result deals with the commuting probability in Lie-rings $L$. Of course, for such rings the concepts of $\mathrm{cp}(L)$ and $\mathrm{zp}(L)$ coincide.

 \begin{thm}\label{thm:commuting}
    Let $L$ be a finite Lie-ring. Suppose that $\mathrm{cp}(L)=\varepsilon$. Then there exists an ideal $D$ in $L$ with $\varepsilon$-bounded index and with $[D,D]$ of $\varepsilon$-bounded order; i.e.\ there exists a function $f$ such that $\lvert [D,D] \rvert \le f(\varepsilon)$ and $[L:D ] \le f(\varepsilon)$.
 \end{thm}

Let $R$ be a ring. We can define on $R$ the structure of a Lie-ring via the operation $[x,y] = xy-yx$. Applying the previous theorem to the Lie-ring $(R,[\cdot,\cdot])$ we immediately obtain the following, which is P.\ Neumann's theorem for finite rings.

 \begin{cor}\label{cor:commuting}
  Let $R$ be a finite ring. Suppose that $\mathrm{cp}(R)=\varepsilon$. Then there exists a Lie-ideal $D$ in the Lie-ring $(R,[\cdot,\cdot])$ with $\varepsilon$-bounded index and with $[D,D]$ of $\varepsilon$-bounded order; i.e.\ there exists a function $f$ such that $\lvert [D,D] \rvert \le f(\varepsilon)$ and $[R:D] \le f(\varepsilon)$.
 \end{cor}

Next, we prove a similar theorem for the probability that the product of two elements in a ring is zero. We say that a ring is a \emph{zero ring} if $R^2=0$, i.e.\ the product of any two elements of $R$ is zero.

   \begin{thm}\label{thm:zero}
Let $R$ be a finite ring such that $\mathrm{zp}(R)=\varepsilon$. Then there exists a two-sided ideal $D$ in $R$ with $\varepsilon$-bounded index and $D^2$ of $\varepsilon$-bounded order; i.e.\ there exists a function $f$ such that $\lvert D^2 \rvert \le f(\varepsilon)$ and $[R:D] \le f(\varepsilon)$.
 \end{thm}

Here $D^2$ denotes the set of all products $d_1d_2$, where both factors are in $D$. It is easy to see that the additive subgroup generated by $D^2$ is a two-sided ideal of $R$ having $\varepsilon$-bounded order.

 We note that the above theorems admit a converse: if the structure of $R$ is as in \ref{cor:commuting} and \ref{thm:zero} (or $L$ is as in \ref{thm:commuting}), then $\mathrm{cp}(R)$ and $\mathrm{zp}(R)$ are bounded away from $0$.

Finally, we remark that the blueprints of the proofs of Theorem~\ref{thm:commuting} and Theorem~\ref{thm:zero} are the same, changing centralizers to annihilators. However, we chose to include both proofs for more clarity. Additionally, we note that the functions appearing in the theorems could be made explicit. It would be interesting to calculate best possible bounds for these functions.

 \section{Commuting probability}

The well-known theorem due to B.\ H.\ Neumann says that if $G$ is a group such that the index of $C_G(x)$ is at most $n$ for all $x$ in $G$, then the commutator subgroup $G'$ has finite $n$-bounded order (see \cite{bneumann, wiegold}). We will now establish an analogous fact for Lie-rings. 

\begin{prop}\label{prop:bounded_square_lie}
Let $L$ be a Lie-ring such that the index of $C_L(x)$ in $L$ is at most $n$ for all $x \in L$. Then $[L,L]$ has finite $n$-bounded order.
\end{prop}
\begin{proof}
Let $a \in L$ be an element such that the index of $C_L(a)$ in $L$ is maximal. We assume that the index equals $n$. Let $b_1, \ldots, b_n \in L$ be elements such that 
$[L,a] = \{[b_1, a], \ldots, [b_n, a]\}$. Let $C = C_L(b_1, \ldots, b_n)$. We will show that, $[L, C] \le [L, a]$. In fact, if $x\in C$, then $[L,a+x]$ contains $\{[b_1, a],\ldots,[b_n, a]\}$. On the other hand, by maximality of $a$, $[L,a+x]$ cannot have more than $n$ elements. Hence, $[L,a+x]= [L,a]$. Thus,
\[
  [L,x] \subseteq [L, a] + [L, a+x] \subseteq [L, a],
\]
as claimed. Let $a_1, \ldots, a_s$ be a transversal of $C$ in the additive group of $L$, and let $a_0 = a$. Note that, since the index of $C_L(b_i)$ is at most $n$, the index $s= \lvert L :C\rvert$ is at most $n^n$. It follows that $$[L,L] = [L,(C+a)+(C+a_1)+\ldots + (C+a_1)] = [L,C] + \sum_{i=0}^s [L,a_i] \le \sum_{i=0}^s [L, a_i].$$ By hypothesis, it follows that $[L,L]$ has $n$-bounded order.
\end{proof}

 In what follows we require the following lemma by Eberhard taken from \cite{eberhard}.

\begin{lem}{\cite[Lemma~2.1]{eberhard}}\label{lem:eberhard}
Let $G$ be a finite group and $X$ a symmetric subset of $G$ containing the identity. Then, $\langle X\rangle = X^{3r}$ provided $(r + 1)\lvert X\rvert > \lvert G\rvert$.
\end{lem}

\begin{proof}[Proof of Theorem~\ref{thm:commuting}]
 In what follows we write $\langle Y \rangle$ to denote the subgroup of the additive group of a Lie-ring $L$ generated by a subset $Y \subseteq L$. Suppose $L$ is a finite Lie-ring such that $\mathrm{cp}(L) = \varepsilon$. Define
 $$X = \{x \in L \mid  \lvert [L, x]\rvert \le 2/\varepsilon\}.$$
Note that $L \smallsetminus X = \{x \in L \mid \lvert C_L(x)\rvert \le (\varepsilon/2)\lvert L\rvert\}$, whence
\begin{multline*}
\varepsilon \lvert L \rvert^2 = \lvert \{(x, y) \in L \times L \mid [x, y] = 0\}\rvert  = \sum_{x\in L} \lvert C_L(x)\rvert
\le \sum_{x\in X} \lvert L \rvert + \sum_{x\notin X} \frac{\varepsilon}{2} \lvert L\rvert \\ \le \lvert X \rvert \lvert L\rvert + (\lvert L \rvert - \lvert X\rvert) \frac{\varepsilon}{2} \lvert L\rvert.
\end{multline*}
Therefore $\varepsilon \lvert L\rvert \le \lvert X\rvert + (\varepsilon/2)( \lvert L\rvert - \lvert X\rvert)$, whence $(\varepsilon/2)\lvert L\rvert < \lvert X\rvert$.

Let $B$ be the additive group generated by $X$. Clearly, $\lvert B\rvert \ge \lvert X \rvert > (\varepsilon/2)\lvert L\rvert$ and so the index of $B$ in $L$ is at most $2/\varepsilon$. As $X$ is symmetric and $(2/\varepsilon)\lvert X\rvert > \lvert L\rvert$, it follows from Lemma~\ref{lem:eberhard} that every element of $B$ is a sum of at most $6/\varepsilon$ elements of $X$. Therefore
$\lvert [L, b]\rvert \le (2/\varepsilon)^{6/\varepsilon}$ for every $b \in B$. Let $D$ be the ideal of $L$ generated by
$B$. Since the index of $B$ in $L$ is $\varepsilon$-bounded, there are boundedly many elements $b_1,\ldots, b_s \in B$ such that $D = B + \sum [L, b_i]$. Since $C_L(b_1, \ldots, b_s)$ normalizes each $[L, b_i]$ and since $[L, b_i]$ has $\varepsilon$-bounded order, we conclude that $C_L(\sum [L, b_i])$ has $\varepsilon$-bounded index in $L$. It follows that $\lvert [L, d]\rvert$ is $\varepsilon$-bounded for every $d\in D$. Thus, by Proposition~\ref{prop:bounded_square_lie}, the theorem follows.
\end{proof}

\section{Product-zero probability}

Throughout, by the annihilator of a subset $S$ of a ring $R$ we mean the \emph{right-annihilator}, i.e.\ $\mathrm{Ann}(S)=\{y\in R \mid sy =0,\ \text{for all } s\in S \}$.

\begin{prop}\label{prop:bounded_square_ring}
Let $R$ be a ring such that the index of $\mathrm{Ann}(x)$ in $R$ is at most $n$ for all $x \in R$. Then $R^2$ has $n$-bounded order.
\end{prop}
\begin{proof}
Let $a \in R$ be an element such that the index of $\mathrm{Ann}(a)$ in $R$ is maximal. We assume that the index equals $n$. Let $b_1, \ldots, b_n \in R$ be elements such that 
$a R=\{a b_1, \ldots, a b_n\}$. Let $C = \mathrm{Ann}(b_1, \ldots, b_n)$. We will show now that, $ C R \le a R$. In fact, if $x\in C$, then $(a+x)R$ contains $\{a b_1,\ldots,a b_n\}$. On the other hand, by maximality of $a$, $(a+x)R$ cannot have more than $n$ elements. Hence, $(a+x)R = aR$. Thus,
\[
  x R \subseteq  aR +  (a+x)R \subseteq  aR,
\]
as claimed. Let $a_1, \ldots, a_s$ be a transversal of $C$ in the additive group of $R$, and let $a_0 = a$. Note that, since $\mathrm{Ann}(b_i)$ has index at most $n$, the index $s=[R:C]$ is at most $n^n$. We see that $$R^2 = \left(C+(C+a_1 )+ \ldots + (C+a_s)\right) R = C R + \sum_{i=1}^s  a_i R \le \sum_{i=0}^s  a_i R.$$ By hypothesis, it follows that $R^2$ has $n$-bounded order.
\end{proof}



\begin{lem}\label{lem:two_sided}
 Let $B$ be a one-sided ideal of a finite ring $R$ such that $[R:B]$ and $|B^2|$ are both at most $n$. Then, there is a two-sided ideal $A$ such that $[R:A]$ and $|A^2|$ are both $n$-bounded.
\begin{proof}
Without loss of generality, we can suppose that $B$ is a right-ideal. We use induction on the index $[R:B]$, which is at most $n$ by hypothesis. The case $[R:B]=1$ is obvious so we assume that $[R:B]\geq2$. By hypothesis $Bx\subseteq B$ for any $x\in R$. Note that whenever $y\in R$, we have $yBx\subseteq yB$ for any $x\in R$ and so $yB$ is a right-ideal, too. If $B$ is a two-sided ideal, we have nothing to prove so assume that there is $y\in B$ such that $yB\not\subseteq B$. Set $D=B+yB$ and note that $D$ is a right-ideal whose index is strictly less than that of $B$. We need to show that $|D^2|$ is $n$-bounded.

Let $d=y b_1 +b_2$ be an arbitrary element of $D$, where $b_1,b_2\in B$. Note that $\mathrm{Ann}(b_1,b_2)\subseteq \mathrm{Ann}(d)$ so we deduce that the index of $\mathrm{Ann}(d)$ is $n$-bounded (at most $n^4$). Thus, by Proposition~3.1, $|D^2|$ is $n$-bounded. Since $D$ is a right-ideal whose index is strictly less than that of $B$, the result follows by induction.
\end{proof}
\end{lem}

We proved the previous lemma for right-ideals to keep the proof consistent with our notation, because we only work with right-annihilators. However, the proof is exactly the same for left-ideals, using left-annihilators (note that the proof of Proposition~\ref{prop:bounded_square_ring} is also symmetric in this sense). Actually, Proposition~\ref{prop:bounded_square_ring} shows that right-annihilators have $n$-bounded index in a ring if and only if left-annihilators have $n$-bounded index, which might be of independent interest.

\begin{proof}[Proof of Theorem~\ref{thm:zero}]
  Suppose $R$ is a finite ring such that $\mathrm{zp}(R) = \varepsilon$. Define
 $$X = \{x \in R \mid  \lvert  xR\rvert \le 2/\varepsilon\}.$$
Note that $R \smallsetminus X = \{x \in R \mid \lvert \mathrm{Ann}(x)\rvert \le (\varepsilon/2)\lvert R\rvert\}$, whence
\begin{multline*}
\varepsilon \lvert R \rvert^2 = \lvert \{(x, y) \in R \times R \mid x y = 0\}\rvert  = \sum_{x\in R} \lvert \mathrm{Ann}(x)\rvert
\le \sum_{x\in X} \lvert R \rvert + \sum_{x\notin X} \frac{\varepsilon}{2} \lvert R\rvert \\ \le \lvert X \rvert \lvert R\rvert + (\lvert R \rvert - \lvert X\rvert) \frac{\varepsilon}{2} \lvert R\rvert.
\end{multline*}
Therefore $\varepsilon \lvert R\rvert \le \lvert X\rvert + (\varepsilon/2)( \lvert R\rvert - \lvert X\rvert)$, whence $(\varepsilon/2)\lvert R\rvert < \lvert X\rvert$.

Let $B$ be the additive group generated by $X$. Clearly, $\lvert B\rvert \ge \lvert X \rvert > (\varepsilon/2)\lvert R\rvert$ and so the index of $B$ in $R$ is at most $2/\varepsilon$. As $X$ is symmetric and $(2/\varepsilon)\lvert X\rvert > \lvert R\rvert$, it follows from Lemma~\ref{lem:eberhard} that every element of $B$ is a sum of at most $6/\varepsilon$ elements of $X$. Therefore
$\lvert bR\rvert \le (2/\varepsilon)^{6/\varepsilon}$ for every $b \in B$.

Let $D$ be the left-ideal generated by $B$. Clearly, $[R:D]$ is $\varepsilon$-bounded. Hence, there are $\varepsilon$-boundedly many elements $b_1,\ldots,b_s\in B$ such that $D = B+ \sum_{i=1}^s R b_i$. Note that $C=\mathrm{Ann}(b_1,\ldots,b_s)$ annihilates all $B b_i$.
Now, for any $y\in D$, if we write $y = b+ a_1 b_1 + \ldots a_s b_s$ for $b\in B$ and $a_i\in R$, 
we have that $$yC = bC + a_1 b_1 C + ... +a_s b_s C = bC.$$
Therefore $\mathrm{Ann}(y)$ contains $\mathrm{Ann}(b) \cap C$, which has $\varepsilon$-bounded index. Thus, $\mathrm{Ann}(y)$ has $\varepsilon$-bounded index for all $y\in D$.
Applying Proposition~\ref{prop:bounded_square_ring} and Lemma~\ref{lem:two_sided}, the theorem follows.

\end{proof}


\begin{thebibliography}{10}

\bibitem{BM}
Stephen~M. Buckley and Desmond MacHale.
\newblock Commuting probability for subrings and quotient rings.
\newblock {\em J. Algebra Comb. Discrete Struct. Appl.}, 4(2):189--196, 2017.

\bibitem{D}
David Dol\v{z}an.
\newblock The probability of zero multiplication in finite rings.
\newblock {\em Bull. Aust. Math. Soc.}, 106(1):83--88, 2022.

\bibitem{eberhard}
Sean Eberhard.
\newblock Commuting probabilities of finite groups.
\newblock {\em Bull. Lond. Math. Soc.}, 47(5):796--808, 2015.

\bibitem{EJ1}
M.~A. Esmkhani and S.~M. Jafarian~Amiri.
\newblock The probability that the multiplication of two ring elements is zero.
\newblock {\em J. Algebra Appl.}, 17(3):1850054, 9, 2018.

\bibitem{EJ2}
M.~A. Esmkhani and S.~M. Jafarian~Amiri.
\newblock Characterization of rings with nullity degree at least {$\frac14$}.
\newblock {\em J. Algebra Appl.}, 18(4):1950076, 8, 2019.

\bibitem{gustafson}
W.~H. Gustafson.
\newblock What is the probability that two group elements commute?
\newblock {\em Amer. Math. Monthly}, 80:1031--1034, 1973.

\bibitem{machale}
Desmond MacHale.
\newblock Commutativity in finite rings.
\newblock {\em Amer. Math. Monthly}, 83(1):30--32, 1975.

\bibitem{bneumann}
B.~H. Neumann.
\newblock Groups covered by permutable subsets.
\newblock {\em J. London Math. Soc.}, 29:236--248, 1954.

\bibitem{neumann}
Peter~M. Neumann.
\newblock Two combinatorial problems in group theory.
\newblock {\em Bull. London Math. Soc.}, 21(5):456--458, 1989.

\bibitem{wiegold}
J.~Wiegold.
\newblock Groups with boundedly finite classes of conjugate elements.
\newblock {\em Proc. Roy. Soc. London Ser. A}, 238:389--401, 1957.

\end{thebibliography}
\end{document}